\documentclass[11pt,leqno]{article}
\usepackage{amssymb,amsfonts,amsmath,amsthm}
\usepackage{graphicx}
\usepackage{epsfig}
\usepackage{float}
\usepackage{subfigure}
\usepackage{hyperref}
\hypersetup{
    colorlinks,%
    citecolor=blue,%
    filecolor=black,%
    linkcolor=black,%
    urlcolor=blue
   }

\newtheorem{theorem}{Theorem}[section]
\newtheorem{lemma}{Lemma}[section]

\numberwithin{equation}{section} \numberwithin{theorem}{section}

\def\bt{\begin{thm}}
\def\et{\end{thm}}

\def\bl{\begin{lem}}
\def\el{\end{lem}}

\def\bd{\begin{defi}}
\def\ed{\end{defi}}

\def\bc{\begin{cor}}
\def\ec{\end{cor}}

\def\bp{\begin{proof}}
\def\ep{\end{proof}}

\def\br{\begin{rem}}
\def\er{\end{rem}}

\newtheorem{thm}{Theorem}[section]
\newtheorem{lem}{Lemma}[section]
\newtheorem{defi}{Definition}[section]
\newtheorem{rem}{Remark}[section]
\newtheorem{cor}{Corollary}[section]

\newcommand{\be}{\begin{equation}}
\newcommand{\ee}{\end{equation}}

\renewcommand{\det}{\mathrm{det}}

\newcommand{\mA}{\mathcal A}

\newcommand{\bif}{\mathfrak{b}_{\lambda}}

\newcommand{\bifc}{\mathfrak{b}_{\lambda_0}}
\newcommand{\I}{\mathcal{I}}
\newcommand{\C}{\mathcal{C}}
\graphicspath{{figs_MT_PDE/}}
\renewcommand{\arraystretch}{1.5}

\begin{document}
\title{Phase Transition Analysis of the Dynamic Instability of Microtubules}
\author{{\sc \small
Shantia Yarahmadian\thanks{Mississippi State University, MS 39762;
syarahmadian@math.msstate.edu} and
Masoud Yari\thanks{Texas A\&M University-Corpus Christi, TX 78412-5825; myari@tamucc.edu} 
}}

\maketitle

\begin{abstract}
This paper provides the phase transition analysis of a reaction diffusion equations system modeling dynamic instability of microtubules. For this purpose we have generalized the macroscopic model studied by Mour\~{a}o et all \cite{MSS}. This model investigates the interaction between the microtubule nucleation, essential dynamics parameters and extinction and their impact on the stability of the system. The considered framework encompasses a system of partial differential equations for the elongation and shortening of microtubules,  where the rates of elongation as well as the lifetimes of the elongating shortening phases are linear functions of GTP-tubulin concentration. In a novel way, this paper investigates the stability analysis and provides a bifurcation analysis for the dynamic instability of microtubules in the presence of diffusion and all of the fundamental dynamics parameters. Our stability analysis introduces the phase transition method  as a new mathematical tool in the study of microtubule dynamics. The mathematical tools introduced to handle the problem should be of general use. 
\end{abstract}

\tableofcontents


\section{Introduction}
Microtubules (MTs) are natural rigid structural  polymers constructed from subunits made of the protein tubulin \cite{DVMW}. They exist inside of living eukaryotic cells and participate in many functions of the cell, including mitosis \cite{AJLR}, axon formation in neurons \cite{SO}, and signaling \cite{EM}. 
MTs are also play an important role in important diseases such as Alzheimer's disease \cite{BVT} and Parkinson's disease \cite{F}. MTs within cells are typically initiated from a complex of proteins that form a nucleation template. Polymerization proceeds at a rate that is dependent upon the concentration of free tubulin subunits \cite{MSS}. Over time, Mts stochastically switching between states of polymerization and depolymerization, a process termed {\it dynamic instability} \cite{SKE, MK, DM}. Switching frequencies are regulated during certain cellular transitions to modify the length distribution and density of the polymer array. The mechanisms governing dynamic instability are still an active subject of both experimental and theoretical research\cite{YBZS}.
\\
There are several analytical models describing how the major factors leading to dynamic instability (i.e. growth and shortening velocities together with catastrophe and rescue frequencies) will create a steady state system of polymers under various conditions \cite{YBZS}. In this work, we advance these analysis of MT dynamics by extending the analytical methods to go through bifurcation analysis associated with the fundamental dynamics parameter, nucleation, extinction and diffusion. The phase transition is investigated in greater detail with particular attention to bifurcation points and dynamic parameters.
\\
Mour\~{a}o et all \cite{MSS} studied the interaction between the microtubule nucleation and dynamics parameters, using macroscopic Monte Carlo simulations, to study the contribution of these parameters in the underlying microtubule array morphology (i.e. polymer density and length distribution).
They found that in addition to the well-characterized steady state achieved between free tubulin subunits and microtubule polymer, microtubule nucleation and extinction constitute a second, interdependent steady state \cite{MSS}. Their studies also shows that the magnitude of both nucleation and extinction rate additively impacts the final steady state free subunit concentration and consequently the nucleation template number plays a defining role in shaping the microtubule length distribution and polymer density \cite{MSS}.
\\
Tubulin diffusion is slow and this slowness causes a fast dephasing in the growth dynamics, unbounded growth of some MTs, and morphological change toward creating bounded short MTs in the nucleation center and unbounded long MTs with narrowly distributed lengths \cite{DYH}. The competition between the rate of the tubulin assembly and the tubulins diffusion rate characterized the transition from unbounded to bounded growth. The present study considers the impact of the tubulin diffusion coefficient along with other dynamic parameters on the dynamic instability.
\\
In this paper, we will study the formation and dynamic instability of MTs by modifying the macroscopic mathematical model of Mour\~{a}o et all. We will consider generalized concentration-dependent model for microtubule growth and shrinkage and includes the diffusion effect as well.

There are in fact three types of phase transition: continues (Type I), jump (Type II) and mix (Type III). The mathematical framework that we will follow here is established in recent work of Ma and Wang (see \cite{MW-B2}). Our study of the asymptotic behavior of the solutions around a bifurcation point is based on a new center manifold reduction procedure developed in \cite{MW-B1} and \cite{KWY}. The key element here is a more precise approximation of the reduced center manifold equations. This method will furnish us with a comprehensive set of reduced equations and certain transition numbers. The behavior of transition numbers will determine  the whole dynamic of the system near the critical bifurcation parameter.

The paper is organized as follows: the mathematical model is presented in Section \ref{sec.2}. In Section \ref{sec.3}, the formation process is discussed and some classic conclusions are derived. Our main results are presented and proved in Section \ref{sec.4}, where we derived the phase transition properties in two main scenarios.

\section{Macroscopic Model for Microtubule Dynamics with Diffusion}\label{sec.2}
Following Mour\~{a}o et all's macroscopic model for MT dynamics, we
consider the following generalized concentration-dependent model for microtubule
growth and shrinkage which includes the diffusion effect:

\begin{equation}\label{ode 1}
\left\{ \begin{array}{l}
\frac{\partial M_g}{\partial t}=-P_{gs}M_g+P_{sg}M_s+N +D_1\Delta M_g\\\\
\frac{\partial M_s}{\partial t}=P_{gs} M_g-P_{sg}M_s-E+D_2\Delta M_s\\\\
\frac{\partial D_f}{\partial t}=-V_gM_g+V_sM_s-N+E+D_3\Delta D_f
\end{array} \right.
\end{equation}
\\
In these equations $M_g$ and $M_s$ represent the number of growing/shrinking microtubules; $D_f$ is the free tubulin concentration. The constants $N$ and $E$ show the nucleation and extinction rate and in general it is assumed that they are  linearly dependent on the concentration of free proteins $D_f$.  The two frequencies $P_{gs}$ and $P_{sg}$ are catastrophe and rescue frequencies and are linearly dependent on $D_f$, i.e. $P_{gs}=-k_7D_f$ and $P_{sg}=k_5D_f$, where $k_5$ and $k_7$ are constants. The two rates $V_g$ and $V_s$ stand for growth and shrinkage rate with linear dependency on tubulin concentration, i.e., $V_g=k_3D_f$ and $V_s=C_1$ where $k_3$ is a constant \cite{MSS}. $D_1$, $D_2$ and $D_3$ are the diffusion rates.
Substituting of these parameters in \ref{ode 1}, will result the following system:

\begin{equation}\label{eq.pde}
\left\{ \begin{array}{l}
\frac{\partial M_g}{\partial t}=-k_7 D_f M_g+k_5 D_f M_s+k_1D_f +D_1\Delta M_g\\\\
\frac{\partial M_s}{\partial t}=k_7 D_f M_g-k_5 D_f M_s-E+D_2\Delta M_s\\\\
\frac{\partial D_f}{\partial t}=-k_3 D_fM_g+V_sM_s-k_1D_f+E+D_3\Delta D_f
\end{array} \right.
\end{equation}

Therefore we consider the system \eqref{eq.pde}
on the spatial domain $\Omega=(0,\ell)$; and we let
\begin{equation}\label{def.f}
\begin{split}
  {{f}_{1}}=&{{f}_{1}}({{M}_{g}},{{M}_{s}},{{D}_{f}})={{k}_{7}}{{D}_{f}}{{M}_{g}}+{{k}_{5}}{{D}_{f}}{{M}_{s}}+{{k}_{1}}{{D}_{f}}, \\
  {{f}_{2}}=&{{f}_{2}}({{M}_{g}},{{M}_{s}},{{D}_{f}})=-{{k}_{7}}{{D}_{f}}{{M}_{g}}-{{k}_{5}}{{D}_{f}}{{M}_{s}}-E, \\
 {{f}_{3}}=&{{f}_{3}}({{M}_{g}},{{M}_{s}},{{D}_{f}})=-{{k}_{3}}{{D}_{f}}{{M}_{g}}+{{C}_{1}}{{M}_{s}}-{{k}_{1}}{{D}_{f}}+E.
\end{split}
\end{equation}
We also assume the system \eqref{eq.pde} is supplemented with usual initial condition and either  Dirichlet boundary conditions
\begin{equation}\label{def.dbc}
  ({M}_{g},{M}_{s},{D}_{f})(0)=({M}_{g},{M}_{s},{D}_{f})(\ell)=0;
\end{equation}
or Neumann boundary conditions
\begin{equation}\label{def.nbc}
   \frac{\partial}{\partial x} ({M}_{g},{M}_{s},{D}_{f})(0)=\frac{\partial}{\partial x}({M}_{g},{M}_{s},{D}_{f})(\ell)=0.
\end{equation}

\section{Early Formation Stage}\label{sec.3}
Formation of microtubule is the result of a balance between different components of the system \eqref{eq.pde}. Like many other pattern forming systems, one naturally assumes that this stage of equilibrium has merged from an earlier state as a result of an initial disequilibrium.
The early behavior of the system can be finely approximated by the linearized system of \eqref{eq.pde}. We therefore will focus here on the linearized system considering the background uniform steady state solution (\ref{def.ss}).
It is easy to see that
\begin{equation}\label{def.ss}
{{M}_{g}}=\frac{k_{1}^{2}{{C}_{1}}}{{{K}_{1}}},{{M}_{s}}=\frac{{{k}_{1}}{{k}_{3}}E}{{{K}_{1}}},
{{D}_{f}}=\frac{E}{{{k}_{1}}},
\end{equation}
where ${{K}_{1}}={{C}_{1}}{{k}_{1}}{{k}_{7}}-{{k}_{3}}{{k}_{5}}E$, is the uniform steady state solution of the equation \ref{eq.pde}. This steady state solution of the system remains always positive if all the parameters are positive and also
\begin{equation}\label{cond.0}
K_1>0.
\end{equation}
We will keep this assumption throughout this paper since we are interested in the realistic situation.

We note that the system of Eqs.~\eqref{eq.pde} defines an abstract evolution equation for a vector-valued function
\[w:t\mapsto w(t)\in H={{({{L}^{2}}(\Omega ))}^{3}}.\]
\begin{equation}\label{def.w}
w(t)=\left( \begin{matrix}
   {{M}_{g}}(\cdot ,t)  \\
   {{M}_{s}}(\cdot ,t)  \\
   {{D}_{f}}(\cdot ,t)  \\
\end{matrix} \right),\quad t\ge 0.
\end{equation}

Now let $\delta:\text{dom}(\delta)\to H$ be defined by the expression	
\begin{equation}\label{def.delta}
\delta =-\Delta I=\left( \begin{matrix}
   -\Delta  & 0 & 0  \\
   0 & -\Delta  & 0  \\
   0 & 0 & -\Delta   \\
\end{matrix} \right).
\end{equation}
where
\begin{equation}\label{}
\text{dom}(\delta)={{H}_{1}}=\{w\in {{({{H}^{2}}(\Omega ))}^{3}}:w=0~\text{on}~\partial \Omega \}
\end{equation}
when \eqref{eq.pde} is considered with \eqref{def.dbc}. Note that $H^2 (\Omega)$ is the usual Sobolev space.

Here we define a new set of conditions, called the zero average condition, as follows
\begin{equation}\label{def.za}
    \int_{\Omega} w_i=0~\text{for}~i=1,2,3;
\end{equation}
and we consider \eqref{eq.pde}  with \eqref{def.nbc} and \eqref{def.za}:
\begin{equation}\label{}
\text{dom}(\delta)={{H}_{1}}=\{w\in {{({{H}^{2}}(\Omega ))}^{3}}:\frac{\partial w_i}{\partial x}=0~\text{on}~\partial \Omega~\text{and}~\int_{\Omega} w_i=0~\text{for}~i=1,2,3~\}.
\end{equation}

Also let $A$ and $D$ be the linear operators in $H$ represented by the constant matrices

\begin{equation}\label{def-AD}
\renewcommand{\arraystretch}{1.5}
A= \begin{pmatrix}
   -\frac{{{k}_{7}}E}{{{k}_{1}}} & \frac{{{k}_{5}}E}{{{k}_{1}}} & 0  \\
   \frac{{{k}_{7}}E}{{{k}_{1}}} & -\frac{{{k}_{5}}E}{{{k}_{1}}} & {{k}_{1}}  \\
   -\frac{{{k}_{3}}E}{{{k}_{1}}} & {{C}_{1}} & -{{K}_{2}}
\end{pmatrix} ,\quad D= \begin{pmatrix}
   {{d}_{1}} & 0 & 0  \\
   0 & {{d}_{2}} & 0  \\
   0 & 0 & {{d}_{3}}  \\
\end{pmatrix} ,
\end{equation}
with ${{K}_{2}}={{k}_{1}}\left( 1+\frac{{{C}_{1}}{{k}_{1}}{{k}_{3}}}{{{K}_{1}}} \right)>0.$
Without loss of generality we can assume ${{k}_{1}}=1,E=1.$  And we let $\kappa =({{k}_{3}},{{k}_{5}},{{k}_{7}})$
and $d=(d_1,d_2,d_3)$.

Now the linearized equation looks like
	\[{{w}_{t}}={{L}_{\lambda }}w.\]
where
\begin{equation}\label{def.L}
{{L}_{\lambda }}=-\delta D+A.
\end{equation}
where $\lambda =(\kappa ,d)$ presents all the control parameters.
Next, let $F={{F}_{\lambda }}:H\to H$ represent the nonlinear terms of the Taylor expansion of $f=({{f}_{1}},{{f}_{2}},{{f}_{3}})$ about the steady state (\ref{def.ss}); that is
\begin{equation}\label{def.F}
\renewcommand{\arraystretch}{1.5}
F(w)=\left( \begin{matrix}
   {{k}_{5}}{{M}_{s}}{{D}_{f}}-{{k}_{7}}{{M}_{g}}{{D}_{f}}  \\
   -{{k}_{5}}{{M}_{s}}{{D}_{f}}+{{k}_{7}}{{M}_{g}}{{D}_{f}}  \\
   -{{k}_{3}}{{M}_{g}}{{D}_{f}}  \\
\end{matrix} \right).
\end{equation}

Finally, then Eq.~(\ref{eq.pde}) corresponds to the abstract evolution equation
\begin{equation}\label{eq.w}
\frac{dw}{dt}={{L}_{\lambda }}w+{{F}_{\lambda }}(w)
\end{equation}
for $w(t)$ in $H$, $t > 0$.

\subsection{Onset of Instabilities}

It is logical to analyze the early behavior of the system by looking at conditions which make some Fourier modes of the solution of \eqref{eq.pde} to become unstable.
Here we will calculate the eigenvalues and eigenfunctions of ${{L}_{\lambda }}$in $H_1$. To this end, assume $({{\rho }_{m}},{{e}_{m}})$ denote the solution to the eigenvalue problem
	\[-\Delta e=\rho e.\]
Then the eigenvalues of ${{L}_{\lambda }}$ are that of its $m^{th}$-component
\begin{equation}\label{def.Em}
{{E}_{m}}=E(\kappa ,d)=\begin{pmatrix}
   -{{k}_{7}}-{{d}_{1}}{{\rho }_{m}} & {{k}_{5}} & 0  \\
   {{k}_{7}} & -{{k}_{5}}-{{d}_{2}}{{\rho }_{m}} & 1  \\
   -{{k}_{3}} & {{C}_{1}} & -{{K}_{2}}-{{d}_{3}}{{\rho }_{m}}  \\
\end{pmatrix}.
\end{equation}
Assume ${{\sigma }_{mi}}$ for $m\in \mathbb{N}$ and $i=1,2,3$ represent eigenvalues of ${{E}_{m}}$; we denote the corresponding eigenvectors by ${{\omega }_{ki}}$. By a straightforward calculation we have
\begin{equation}\label{def.omega}
 {{\omega }_{mi}}=\begin{pmatrix}
   {{k}_{5}}  \\
  {{d}_{1}}\rho_m +{{k}_{7}}+{{\sigma }_{mi}}  \\
   \left({{d}_{1}}\rho_m +{{k}_{7}}+{{\sigma }_{mi}} \right)\left({{d}_{2}}\rho_m +{{k}_{5}}+{{\sigma }_{mi}} \right)-{{k}_{5}}{{k}_{7}}  \\
\end{pmatrix}
\end{equation}
It is obvious that ${{\sigma }_{ki}}$`s are eigenvalues of $L_{\lambda}$ and the corresponding eigenfunctions are given by
	${{w}_{ki}}={{\omega }_{ki}}{{e}_{k}}.$
It is easy to verify that the conjugate operator   has the same eigenvalues with eigenfunctions
	$w_{ki}^{*}=\omega _{ki}^{*}{{e}_{k}},$
where  $\omega _{mi}^{*}$ are eigenvectors of $E_{m}^{*}(\kappa ,d)$, and we have
\begin{equation}\label{def.omegast}
\omega _{mi}^{*}= \begin{pmatrix}
   {{C}_{1}}{{k}_{7}}-\left({{d}_{2}}\rho_m +{{k}_{5}}+{{\sigma }_{m,i}} \right){{k}_{3}}  \\
   \left( {{d}_{1}}\rho_m +{{k}_{7}}+{{\sigma }_{m,i}} \right){{C}_{1}}-{{k}_{3}}{{k}_{5}}  \\
   \left( {{d}_{2}}\rho_m +{{k}_{5}}+{{\sigma }_{m,i}} \right)\left( {{d}_{1}}\rho_m +{{k}_{7}}+{{\sigma }_{m,i}} \right)-{{k}_{5}}{{k}_{7}}
\end{pmatrix}.
\end{equation}

Now we assume
	\[{{\Lambda }_{k}}=\left\{ (\kappa,d) |det({{E}_{k}}(\kappa ,d))=0 \right\}.\]
We note that
	\[\frac{\partial {{\Lambda }_{1}}}{\partial (\kappa ,d)}\ne 0\]
If we assume
\begin{equation}\label{cond.1}
  {{k}_{5}}{{K}_{2}}\ne {{C}_{1}}\text{ and }{{C}_{1}}\ne {{d}_{2}}{{d}_{3}}\rho _{1}^{2}+{{d}_{2}}{{K}_{2}}{{\rho }_{1}}.
\end{equation}
Therefore $\Lambda_1$ divides the space to two distinct regions which we denote by ${{\Lambda }^{-}}$  and $\Lambda^+$. We then have the following lemma:
\begin{lemma}[Exchange of Stability]\label{lemma.pes}
{\rm If the following condition is satisfied
\begin{equation}\label{cond.2}
{{k}_{5}}{{K}_{2}}-{{C}_{1}}>0,
\end{equation}
then we will have
\begin{equation}\label{}
{{\sigma }_{11}}(\kappa ,d)\begin{cases}
<0 & \text{if} \quad (\kappa ,d)\in {{\Lambda }^{-}},\\
 =0 & \text{if} \quad (\kappa ,d)\in {{\Lambda }^{0}},  \\
 >0 & \text{if} \quad (\kappa ,d)\in {{\Lambda }^{+}};
\end{cases}
\end{equation}
and also we have
	\[{{\sigma }_{12}}(\kappa ,d),{{\sigma }_{13}}(\kappa ,d)<0 \text{ if } (\kappa ,d)\in {{\Lambda }^{0}};\]
moreover, for $m>2$ and $i=1,2,3$ we have
	\[{{\sigma }_{mi}}(\kappa ,d)<0 \text{ if } (\kappa ,d)\in {{\Lambda }^{0}}.\]
}\end{lemma}
\begin{proof}
 Assume the characteristic polynomial of   is given by
	\[{{\sigma }^{3}}+p_{2}^{m}{{\sigma }^{2}}+p_{1}^{m}\sigma +p_{0}^{m}=0.\]
We note that \[\underset{i=1}{\overset{3}{\mathop \prod }}\,{{\sigma }_{mi}}=-p_{0}^{m}=\det ({{E}_{m}}).\] It can be easily check that
\begin{equation}\label{cond.tr}
\sum\limits_{i=1}^{3}{{{\sigma }_{mi}}}=-p_{2}^{m}=\text{trace}({{E}_{m}})<0\text{ for all m}\in \mathbb{N}.
\end{equation}
Also, by (\ref{cond.2})  we will have
\begin{equation}\label{cond.dm}
\sum\limits_{\begin{subarray}{c}i,j=1 \\ i\ne j \end{subarray}}^{3}\,
{{\sigma }_{mi}}{{\sigma }_{mj}}=p_{1}^{m}<0\text{  for all m}\in \mathbb{N}.
\end{equation}
Now, on $\Lambda^0$ we have ${\underset{i=1}{\overset{3}\prod}}\,\sigma_{1i}=0$. By (\ref{cond.tr}) and (\ref{cond.dm}) only one of these eigenvalues must be zero on $\Lambda^0$. We denote this critical eigenvalue  by $\sigma_{11}$. Since $\sigma_{11}$ depends continuously on the control parameter, it changes sing when the control parameter crosses $\Lambda^0$. We denoted the two distinct regions by ${{\Lambda }^{-}}$  and $\Lambda^+$. Moreover, by (\ref{cond.tr}) and (\ref{cond.dm}), the other two eigenvalues have to be negative. That is
\[{{\sigma }_{12}}(\kappa ,d)<0,{{\sigma }_{13}}(\kappa ,d)<0 \text{ if } (\kappa ,d)\in {{\Lambda }^{0}}.\]
To prove the last claim, we should note that
\begin{equation}\label{id.Em}
{{E}_{m}}(\kappa ,d)=E({{\rho }_{m}},\kappa ,d)=E({{\rho }_{1}},\kappa ,{{\rho }_{1}}\rho _{m}^{-1}d)={{E}_{1}}(\kappa ,{{\rho }_{1}}\rho _{m}^{-1}d).
\end{equation}
Also by (\ref{cond.C}) the determinant of the matrix $E(\rho )$ is decreasing in $\rho$. Therefore, for $({{\kappa }^{0}},{{d}^{0}})\in {{\Lambda }^{0}}$, we will have
\[\det ({{E}_{m}}({{\kappa }^{0}},{{d}^{0}}))<\det ({{E}_{1}}({{\kappa }^{0}},{{d}^{0}}))=0.\]
Hence from \ref{id.Em} we can conclude
$({{\kappa }^{0}},{{\rho }_{1}}\rho _{m}^{-1}{{d}^{0}})\in {{\Lambda }^{-}}.$
This obviously means that all the eigenvalues of
${{E}_{1}}({{\kappa }^{0}},{{\rho }_{1}}\rho _{m}^{-1}{{d}^{0}})$
are negative; hence all the eigenvalues of
${{E}_{m}}({{\kappa }^{0}},{{d}^{0}})$ for $m>1$
are negative. This completes the proof.
\end{proof}
\section{Transitions}\label{sec.4}
In the previous section we introduced the threshold of instabilities  for the system \eqref{eq.pde}. Lemma \ref{lemma.pes} states that when the control parameter of the system $\lambda$ crosses the instability threshold, the system undergoes a change of stability. When $\lambda$ is near the instability threshold, one Fourier mode of the solution, called the principal mode, becomes unstable. This should eventually lead the solution to a stable situation. However, this is not clear from the linear analysis. It is known that any type of pattern formation is due to the presence of nonlinear components in the system.

In order to study the transitions of the nonlinear system \eqref{eq.pde}, we will study the qualitative behavior of its solutions when $\lambda$ stays close enough to the instability threshold $\Lambda^0$. Consider the system
\begin{equation}\label{eq.w}
 {{w}_{t}}={{L}_{\lambda }}(w)+{{F}_{\lambda }}(w).
\end{equation}
we aim to drive a transition number $\bif$ which will provide comprehensive information about transitional behavior of the solutions at $\lambda$.
Based on the center manifold theorem, for any control parameter $\lambda$ in the vicinity of $\Lambda^0$, we have
\begin{equation}\label{def.y}
w=y{{w}_{11}}+\Phi (y)
\end{equation}
where $w$ is the solution \eqref{eq.pde}, ${{w}_{11}}$ is the principal mode  and $y$ is its amplitude. The error term in fact depends on the leading amplitude $y$.

{\theorem \label{th.dbc}  Consider the system of equations \eqref{eq.w} with \eqref{def.dbc}.
If $\lambda$ is close to the critical parameter $\lambda_0 \in \Lambda^0$, and all components $\lambda_0$ satisfy conditions \eqref{cond.0}, \eqref{cond.1},\eqref{cond.2}, and $d_2>d_1$, then the following assertions hold true
\begin{enumerate}
\item \eqref{eq.w} has a mixed transition from $(0,\lambda_0)$. More precisely, there exists an open neighborhood $U$ of $w = 0$
such that $U$ is separated into two disjoint open sets, ${U^{\lambda}_1}$ and  ${U^{\lambda}_2}$,
by the stable manifold $\Gamma$ of $w=0$ with codimension one in $H$ satisfying

\item \eqref{eq.w} bifurcates from $(0,\lambda_0)$ to a unique
saddle point $w^{\lambda}$ (with Morse index one) on $\lambda \in
\Lambda_{-}$, and to a unique attractor $w^{\lambda}$ on $\lambda \in
\Lambda_+$.
\begin{enumerate}
  \item ${U} = {U^{\lambda}_1} + {U^{\lambda}_2}$
  \item the transition in ${U^{\lambda}_1}$ is jump, and
  \item the transition in ${U^{\lambda}_2}$ is continuous. The local transition structure is as shown in Fig.~\ref{fig.dbc}.
\end{enumerate}
\item \eqref{eq.w} bifurcates in ${U^{\lambda}_2}$
 to a unique singular point $w^{\lambda}$ for $\lambda \in \Lambda^+$, which is an
attractor such that for any $\varphi \in {U^{\lambda}_2}$ we have
\[\lim_{t\rightarrow \infty}||w(t,\varphi) - w^{\lambda}||_H = 0\]
\item \eqref{eq.w} bifurcates for $\lambda \in \Lambda^-$ to a unique
saddle point $w^{\lambda}$ with the Morse index
one.
\item  Near $\lambda_0$, the bifurcated
singular points $w_{\lambda}$ can be expressed as
$$ w^{\lambda}= \tfrac{-\sigma_{11}(\lambda)}{\alpha} w_{11} + o(|\sigma_{11}(\lambda)|). $$
where $w_{11}$ is given in \eqref{def.omega}.
\end{enumerate}}

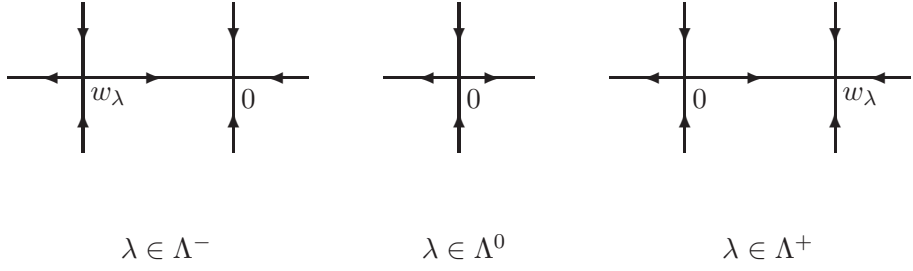
\begin{figure}
\begin{picture}(40,45)
\setlength{\unitlength}{1mm}
\thicklines
\put(0, 20){\line(1,0){40}}
\put(10,10){\line(0,1){20}}
\put(30,10){\line(0,1){20}}
\put (15,-4){$\lambda \in \Lambda^-$}
\put (11,17){$w_\lambda$}
\put (31,16){$0$}
\put(5,20){\vector(-1,0){.5}}
\put(20,20){\vector(1,0){.5}}
\put(35,20){\vector(-1,0){.5}}
\put(10,15){\vector(0,1){.5}}
\put(10,25){\vector(0,-1){.5}}
\put(30,15){\vector(0,1){.5}}
\put(30,25){\vector(0,-1){.5}}

\put(50, 20){\line(1,0){20}}
\put(60,10){\line(0,1){20}}
\put (55,-4){$\lambda \in \Lambda^0$}
\put (61,16){$0$}
\put(60,15){\vector(0,1){.5}}
\put(60,25){\vector(0,-1){.5}}
\put(55,20){\vector(-1,0){.5}}
\put(65,20){\vector(1,0){.5}}

\put(80, 20){\line(1,0){40}}
\put(90,10){\line(0,1){20}}
\put(110,10){\line(0,1){20}}
\put(90,15){\vector(0,1){.5}}
\put(90,25){\vector(0,-1){.5}}
\put(110,15){\vector(0,1){.5}}
\put(110,25){\vector(0,-1){.5}}
\put(85,20){\vector(-1,0){.5}}
\put(100,20){\vector(1,0){.5}}
\put(115,20){\vector(-1,0){.5}}

\put (95,-4){$\lambda \in \Lambda^+$}
\put (91,16){$0$}
\put (111,17){$w_\lambda$}
\end{picture}
\vspace{8pt}
\caption{\footnotesize Topological structure of the dynamic bifurcation
of the \eqref{eq.w} equation with \eqref{def.dbc}. The horizonal line represents the center manifold.}
\label{fig.dbc}
\end{figure}
\begin{proof}[Proof of Th.~\ref{th.dbc}]
For $\lambda$ close to  $\lambda_0 \in \Lambda^0$, we project the system to its first eigenspace; therefore,
	\[\left\langle {{w}_{t}},w_{11}^{*} \right\rangle =\left\langle {{L}_{\lambda }}(w)+{{F}_{\lambda }}(w),w_{11}^{*} \right\rangle =\left\langle {{L}_{\lambda }}(w),w_{11}^{*} \right\rangle +\left\langle {{F}_{\lambda }}(w),w_{11}^{*} \right\rangle .\]
where $\left\langle , \right\rangle$ is the inner product in ${{L}^{2}}(\Omega )$. We note that
	\[\left\langle {{w}_{t}},w_{11}^{*} \right\rangle \text{= }\frac{dy}{dt}\left\langle {{w}_{11}},w_{11}^{*} \right\rangle \text{  and   }\left\langle {{L}_{\lambda }}(w),w_{11}^{*} \right\rangle ={{\sigma }_{11}}(\lambda )\left\langle {{w}_{11}},w_{11}^{*} \right\rangle \]
By the classic center manifold theorem, we can easily see that
	\[\left\langle {{F}_{\lambda }}(w),w_{11}^{*} \right\rangle ={{y}^{2}}\left\langle {{F}_{\lambda }}({{w}_{11}}),w_{11}^{*} \right\rangle +o(|y{{|}^{2}})\]
Therefore we can easily drive the following reduced center manifold equation at $\lambda_0$
	\[\frac{dy}{dt}=\alpha ({{\lambda }_{0}}){{y}^{2}}+o({{y}^{2}});\]
where $\alpha ({{\lambda }_{0}})=\frac{\left\langle {{F}_{{{\lambda }_{0}}}}({{w}_{11}}),w_{11}^{*} \right\rangle }{\left\langle {{w}_{11}},w_{11}^{*} \right\rangle }$ is a constant. It is a straightforward calculation to see
	\[\alpha ({{\lambda }_{0}})=\frac{\left\langle {{F}_{{{\lambda }_{0}}}}({{w}_{11}}),w_{11}^{*} \right\rangle }{\left\langle {{w}_{11}},w_{11}^{*} \right\rangle }=\frac{\int\limits_{0}^{L}{e_{1}^{3}}}{\int\limits_{0}^{L}{e_{1}^{2}}}\frac{{{F}_{{{\lambda }_{0}}}}(\omega )\cdot {{\omega }^{*}}}{\omega \cdot {{\omega }^{*}}}=\frac{8}{3\pi }\frac{{{F}_{{{\lambda }_{0}}}}(\omega )\cdot {{\omega }^{*}}}{\omega \cdot {{\omega }^{*}}}\]
is nonzero (in fact, negative) if ${{d}_{1}}>{{d}_{2}}$. For the sake of simplicity in notations, we
drop the indices and simply write $\omega=\omega_{11}$ and $\omega^*=\omega_{11}^*$.
Now when $\lambda$ lies near $\lambda_0$, the reduced center manifold equation is approximated as
	\[\frac{dy}{dt}={{\sigma }_{11}}(\lambda )y+\alpha (\lambda ){{y}^{2}}+o({{y}^{2}})\]
where  \[\underset{\lambda \to {{\lambda }_{0}}}{\mathop{\lim }}\,\alpha (\lambda )=\alpha ({{\lambda }_{0}})<0.\] If the solutions are regular, one can approximate the solution of the above equation with
	\[\frac{dy}{dt}={{\sigma }_{11}}(\lambda )y+\alpha (\lambda ){{y}^{2}}.\]
A simple examination of the asymptotic behavior of this last equation shows that the only non-trivial steady state solution is
	\[y=-\frac{{{\sigma }_{11}}(\lambda )}{\alpha (\lambda )}\]
which is an attractor when ${{\sigma }_{11}}(\lambda )>0$ and is a repeller  when ${{\sigma }_{11}}(\lambda )<0$.
\end{proof}


When the system \eqref{eq.pde} is considered with the Neumann boundary conditions, driving a feasible expresion from
\[\frac{\left\langle {{F}_{\lambda }}(y{{w}_{11}}+\Phi (y)),w_{11}^{*} \right\rangle }{\left\langle {{w}_{11}},w_{11}^{*} \right\rangle }\]
is not as easy anymore. In this case,
\begin{equation}\label{}
    \left\langle {{F}_{\lambda }}(y{{w}_{11}}),w_{11}^{*} \right\rangle=0;
\end{equation}
therefore, we have to proceed to another stage of higher order approximation. The difficulty lies in driving a good approximation of $\langle{{F}_{\lambda}}(y{{w}_{11}}+\Phi(y))\:,\:w_{11}^{*}\rangle$. To overcome this difficulty, one normally resorts to a simple Taylor approximation of the implicit center manifold function. However, we will use another approximation method suggested by Ma and Wang \cite{mw-b}. Using their approach will lead us to the following approximation
	\[\frac{\left\langle {{F}_{\lambda }}(w),w_{11}^{*} \right\rangle }{\left\langle {{w}_{11}},w_{11}^{*} \right\rangle }=\bif{{y}^{3}}+o(|y{{|}^{3}}).\]
where $\bif$ is the transition number which we will drive later. Therefore the transition equation near the threshold of instabilities reads as
	\[\dot{y}=y{{\sigma }_{11}}+\bif{{y}^{3}}+o(|y{{|}^{3}}).\]
This suggest the existence of a pitchfork bifurcation as the control parameter $\lambda =(\kappa ,d)$ crosses the critical threshold ${{\Lambda }^{0}}$.
We have the following  theorem
\begin{theorem} \label{th.nbc}
Consider the system (\ref{eq.w}) with \eqref{def.nbc} and \eqref{def.za}. If $\lambda$ is close to the critical parameter $\lambda_0 \in \Lambda^0$, and all components of $\lambda_0$ satisfy conditions \eqref{cond.0}, \eqref{cond.1},and \eqref{cond.2}, then
the following assertions hold true
\begin{description}
  \item[Case 1] If $\bifc< 0$, the transition of (\ref{eq.w}) over $\lambda_0$ is continuous (type I); moreover,
  \begin{enumerate}
\item  the trivial solution $w_{\lambda_0} = 0$ is a locally asymptotically stable equilibrium point of the system (\ref{eq.w}) at $\lambda_0$;
\item after bifurcation, the solution $w$ of the system (\ref{eq.w}) will asymptotically tend to either $w_\lambda^+$ or $w_\lambda^-$,
    where
\begin{equation}
  w_\lambda^{\pm}  = {\pm} (\sigma_{k1}/|\bif|)^{1/2} w_{k1}
  + \epsilon_\lambda ,
\label{def.pitch}
\end{equation}
where
$\|\epsilon_\lambda\|_H = o (\sigma_{k1}^{1/2})$.(see Fig.~\ref{fig.pitch})
\end{enumerate} 
  \item[Case 2] If $\bifc > 0$, the transition of (\ref{eq.w}) over $\lambda_0$ is a jump (type II). The steady state solutions are metastable after bifurcation. In this case, points $w_\lambda^+$ and $w_\lambda^-$ will be repellers.
       \item[Case 3] If $\bifc = 0$, the transition of (\ref{eq.w}) over $\lambda_0$ is mixed (type III).
\end{description}
\end{theorem}
\begin{figure}[htb]
\begin{center}
 {\setlength{\unitlength}{0.7mm}
\begin{picture}(60,70)
  \thicklines
  \put(10,0){\line(0,1){63}}
  \put(-15, 30){\line(1,0){25}}
  \put(10,30){\circle*{1}}
  \put(13, 30){\line(1,0){3}}
  \put(19, 30){\line(1,0){3}}
  \put(25, 30){\line(1,0){3}}
  \put(31, 30){\line(1,0){3}}
  \put(37, 30){\line(1,0){3}}
  \put(44,30){\vector(1,0){7}}
  \put(1, 20){\line(0,1){20}}
  \put(1,25){\vector(0,1){.5}}
  \put(1,35){\vector(0,-1){.5}}
  \put(1,30){\circle*{1}}
\qbezier(35.000,50.000)(-15.000,30.000)
         (35.000,10.000)
         \put(25, 0){\line(0,1){60}}
    \put(25,20){\vector(0,-1){.5}}
    \put(25,40){\vector(0,1){.5}}
    \put(25,50){\vector(0,-1){.5}}
    \put(25,10){\vector(0,1){.5}}
    \put(25, 30){\circle{1}}
     \put(25, 14.45){\circle*{1}}
     \put(25, 45.3){\circle*{1}}
     \put(26, 15){$w^{-}$}
     \put(26, 43.5){$w^{+}$}
   \put(8,65){$\Lambda_0$}
   \put(-33,63){Stability Region}
   \put(26,63){Instability Region}
   \put(4,25){$\lambda_0$}
   \put(48,25){$\lambda$}
\end{picture}}
\end{center}
\caption{supercritical bifurcation to an attractor
$\mA = \{w^{+}, w^{-} \}$.}
\label{fig.pitch}
\end{figure}
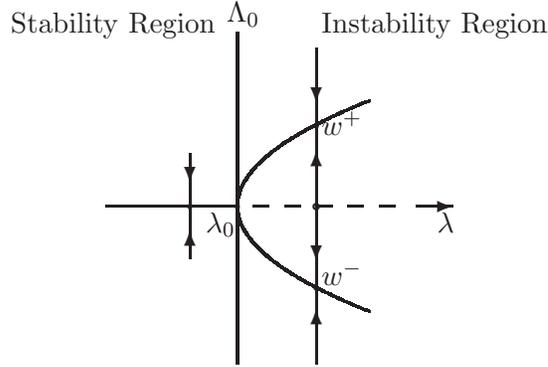
The calculation of the parameter $\bif$ becomes, therefore, very crucial.
This calculation is lengthy and tedious, but it can reveal valuable information
regarding the number of the steady state solution and the type of the bifurcation and transition.

Before we proceed with the proof of Th.\ref{th.nbc}, we state a lemma due to Ma and Wang, but we refer the reader to \cite{MW-B1} for a proof.
\begin{lemma}[Approximation of the center manifold function]\label{lemma.cmf}
Assume $w=\sum\limits_{I\in \C} y_I w_I +\Phi$ is the solution of \ref{eq.w} where $\C$ is the set of critical indices. Define
	\[{{E}_{1}}=\text{span}\{{{w}_{I}}|I\in \C\},E_2=E_{1}^{\bot}, \mathcal{L}=L_{\lambda }|_{E_1};\]
and let $\mathcal{P}_2:H \to E_2$ be the Leray projection. Then
	\[-\mathcal{L}^{-1}(\Phi (y))=\mathcal{P}_2(F(\sum\limits_{I\in \C}y_I w_I))+o(|y|^2)+O(|\sigma_I| |y|^2)\]
\end{lemma}

\begin{proof}[Proof of Th.~\ref{th.nbc}]
We note that
	\[\left\langle {{F}_{\lambda }}(w),w_{11}^{*} \right\rangle =\left\langle {{F}_{\lambda }}(w),{{e}_{1}}\omega _{11}^{*} \right\rangle =\left\langle {{F}_{\lambda }}(w),{{e}_{1}} \right\rangle \cdot \omega _{11}^{*}.\]
We write $\langle {{F}_{\lambda }}(w),{{e}_{1}}\rangle$ just for simplicity, but we really take the ${{L}^{2}}$-inner product of each component of ${{F}_{\lambda }}(w)$ with ${{e}_{1}}$.
We have
	 \[w=y{{w}_{11}}+{{y}_{12}}{{w}_{12}}+{{y}_{13}}{{w}_{13}}+\sum\limits_{\substack{I>1\\i=1..3}}{{{y}_{Ii}}{{w}_{Ii}}}\]
There exist a finite set of indices $\I$ so that we have
\[\langle {{F}_{\lambda }}(w),{{e}_{1}}\rangle =\langle {{F}_{\lambda }}(\sum\limits_{ \substack{I\in \I\\i=1..3}}{{{y}_{Ii}}{{w}_{Ii}}}),{{e}_{1}}\rangle \]

 It is easy to see that here $\I=\{2\}$. For the sake of simplicity in notations, again let $\omega=\omega_{11}$ and $\omega^*=\omega_{11}^*$. Let us define
\[{{B}^{j}}(I):=\left\langle e_{1}^{2},{{e}_{I}} \right\rangle \sum\limits_{i=1}^{3}{b_{I,i}^{j}{{y}_{I,i}}}=\left\langle e_{1}^{2},{{e}_{I}} \right\rangle b_{I}^{j}{{y}_{I}}\quad j=1,2,\]
where
\[b_{I}^{j}=(b_{I,1}^{j},b_{I,2}^{j},b_{I,3}^{j}),\quad {{y}_{I}}={{({{y}_{I1}},{{y}_{I2}},{{y}_{I3}})}^{T}}\]
with
$b_{Ii}^{j}=({{\omega }^{j}}\omega _{Ii}^{3}+{{\omega }^{3}}\omega _{Ii}^{j}).$
Now let us also define
	\[B:=\left[ \begin{matrix}
   {{k}_{5}} {{B}^{2}}( 2 )-{{k}_{7}}{{B}^{1}}( 2)  \\
   -{{k}_{5}} {{B}^{2}}( 2 ) +{{k}_{7}}{{B}^{1}}( 2 )  \\
   -{{k}_{3}}{{B}^{1}}( 2 )  \\
\end{matrix} \right].\]
It is a straightforward calculation to see that
	\[\dot{y}={{B}^{T}}{{\omega }^{*}}y+o(|y{{|}^{3}}).\]

By Lemma~(\ref{lemma.cmf}) we can see that
\begin{equation}\label{}
    \begin{split}
     {{y}_{Ii}}&={{(-{{\sigma }_{Ii}}\left\langle {{w}_{I}},w_{I}^{*} \right\rangle )}^{-1}}\left\langle F(y{{w}_{11}}),w_{Ii}^{*} \right\rangle +o(|y{{|}^{2}}) \\
  &={{(-{{\sigma }_{Ii}}\left\langle {{e}_{I}},{{e}_{I}} \right\rangle {{\omega }_{Ii}}\cdot \omega _{Ii}^{*})}^{-1}}\left\langle e_{1}^{2},{{e}_{I}} \right\rangle F(\omega )\cdot \omega _{Ii}^{*}{{y}^{2}}+o(|y{{|}^{2}})
    \end{split}
\end{equation}

Therefore,
	\[{{y}_{I}}=\frac{-\left\langle e_{1}^{2},{{e}_{I}} \right\rangle }{\left\langle {{e}_{I}},{{e}_{I}} \right\rangle }\text{diag}{{[{{\sigma }_{Ii}}{{\omega }_{Ii}}\cdot \omega _{Ii}^{*}]}^{-1}}{{[\omega _{I1}^{*}\quad \omega _{I2}^{*}\quad \omega _{I3}^{*}]}^{T}}F(\omega ){{y}^{2}}+o(|y{{|}^{2}}).\]
where by $\text{diag}[{{\sigma }_{Ii}}{{\omega }_{Ii}}\cdot \omega _{Ii}^{*}]$ we mean a diagonal matrix with diagonal elements ${{\sigma }_{Ii}}{{\omega }_{Ii}}\cdot \omega _{Ii}^{*}$ for $i=1,2,3$. Consequently, we will have
	\[{{B}^{j}}\left( I \right)=\frac{-{{\left\langle e_{1}^{2},{{e}_{I}} \right\rangle }^{2}}}{\left\langle {{e}_{I}},{{e}_{I}} \right\rangle }b_{I}^{j}\text{diag}{{({{\sigma }_{Ii}}{{\omega }_{Ii}}\cdot \omega _{Ii}^{*})}^{-1}}{{[\omega _{I1}^{*}\quad \omega _{I2}^{*}\quad \omega _{I3}^{*}]}^{T}}F(\omega ){{y}^{2}}+o(|y{{|}^{2}}).\]

Note that
	\[\frac{{{\left\langle e_{1}^{2},{{e}_{I}} \right\rangle }^{2}}}{\left\langle {{e}_{I}},{{e}_{I}} \right\rangle }=\frac{\ell }{8}.\]
and
\[F(\omega )=\left[ \begin{matrix}
   {{k}_{5}}{{\omega }^{2}}{{\omega }^{3}}-{{k}_{7}}{{\omega }^{1}}{{\omega }^{3}}  \\
   -{{k}_{5}}{{\omega }^{2}}{{\omega }^{3}}+{{k}_{7}}{{\omega }^{1}}{{\omega }^{3}}  \\
   -{{k}_{3}}{{\omega }^{1}}{{\omega }^{3}}  \\
\end{matrix} \right].\]
\end{proof}
\begin{rem}\rm
In the proof of Th.~\ref{th.nbc}, the parameter $\bif$ is given as an algebraic expression. However, its expression is large; consequently any quantitative calculation requires a long calculation which is trivial yet tedious. We remark that  $\bif$ can be substantially simplified by assuming
\begin{equation}\label{def.C1}
  {{C}_{1}}{{k}_{7}}={{k}_{3}}( {{k}_{5}}+\rho {{d}_{2}})
\end{equation}
It can be easily verified that under the assumption \eqref{def.C1}, we will have
\begin{equation}\label{}
    \begin{split}
    {{B}^{T}}{{\omega }^{*}}&={{k}_{5}}\left( \omega _{1}^{*}-\omega _{2}^{*} \right)\left( {{B}^{2}}\left( 2 \right)+{{B}^{2}}
    \left( 0 \right) \right) \\
 & =-\frac{{{k}_{3}}{{k}_{5}}\rho}{{{k}_{7}}}\left( {{{k}_{5}}{{d}_{1}}+{k}_{7}}{{d}_{2}}+{{d}_{1}}{{d}_{2}}\rho \right){{B}^{2}}\left( 2 \right).
    \end{split}
\end{equation}
where $\rho ={{\rho }_{11}}$.
\end{rem}
	
%
%
%

\section{Conclusion}
From a biological point of view, the above results show the complex nature of the formation and phase transition of MTs in the competition of nucleation rate, extinction rate and the dynamic instability fundamental growth, shrinkage and stochastic switching process as well as  concentration dependency of the dynamic instability parameters.  These results show also the necessity for a more elaborate multistate models to include the effect of more complex boundary conditions as well as considering a two dimensional diffusion. A numerical analysis of the current study is in progress. In addition, stability analysis of the more complex boundary conditions is in progress. It is our hope that this will open the way to the study of more complicated and realistic biological models with a systematic study across physical parameters.

\section*{Acknowledgments}
The authors would like to thank Prof. Shouhong Wang (Indiana University) for helpful discussions and Henry Family Research Fund (HFRF) in Mississippi State University for partial support of this project.

\end{document}